\documentclass[a4paper ,english]{article}
\usepackage[utf8]{inputenc}
\usepackage{graphicx}
\usepackage{amssymb}
\usepackage{amsmath}
	\allowdisplaybreaks
\usepackage{amsfonts}%
\usepackage[left=30mm,right=30mm, top=3cm, bottom=2cm]{geometry}
\usepackage{booktabs}

\newcommand{\defgl}{\mathrel{\mathop:}=}

\usepackage{listings}
\usepackage{caption}
\usepackage{subcaption}
\usepackage{nameref}
\makeatletter
\let\orgdescriptionlabel\descriptionlabel
\renewcommand*{\descriptionlabel}[1]{%
  \let\orglabel\label
  \let\label\@gobble
  \phantomsection
  \edef\@currentlabel{#1}%
  \let\label\orglabel
  \orgdescriptionlabel{#1}%
}
\makeatother

\usepackage{amsthm}
\newcounter{mycounter}[section]%

\newtheorem{theorem}[mycounter]{Theorem}
\newtheorem{lem}[mycounter]{Lemma}
\newtheorem{method}[mycounter]{Method}
\newtheorem{prop}[mycounter]{Proposition}
\newtheorem{defi}[mycounter]{Definition}
\newtheorem{kor}[mycounter]{Corollary}
\newtheorem{bsp}[mycounter]{Example}
\newtheorem{rem}[mycounter]{Remark}
\newtheorem*{assumpt}{Assumption}
	
\usepackage{numprint}
\usepackage{xcolor}
\definecolor{numberstylegray}{rgb}{0.5,0.5,0.5}%
\definecolor{keywordcolor}{rgb}{0,0,0.75}
\lstdefinestyle{pseudocode}{frame=l, morekeywords={if, try, for, return, else}, keywordstyle=\bfseries\color{keywordcolor}, float = tb, captionpos=b, numbers=left, numberstyle=\footnotesize\color{numberstylegray}, escapeinside={*(}{*)}, xleftmargin=.1\textwidth, xrightmargin=.05\textwidth}
\newcolumntype{R}[1]{>{\raggedleft\arraybackslash}p{#1}}
\title{Topological conditions on inhomogeneous fractals in Martin boundary theory and their algorithmic testing}
\author{Stefan Kohl\footnote{Institute for Stochastics and Applications, University of Stuttgart, Pfaffenwaldring 57, 70569 Stuttgart, Germany, E-mail: \href{mailto:Stefan.Kohl@mathematik.uni-stuttgart.de}{\texttt{Stefan.Kohl@mathematik.uni-stuttgart.de}}}}
\usepackage{hyperref} 
\hypersetup{pdftitle={Topological conditions on inhomogeneous fractals in Martin boundary theory and their algorithmic testing}, pdfauthor={Stefan Kohl}}

\begin{document}
\maketitle
\begin{abstract}
We want to consider Martin boundary theory applied to inhomogeneous fractals. This is under some conditions possible, but up to now it is not clear, how one can easily check, if a certain fractal fulfills those conditions. We want to simplify one condition and further develop a computer algorithm, which can check, if an attractor fulfills the condition. 
\end{abstract}
\bigskip
%
% 
%----------------------------------------------------------------------------------------------------------------------------------------------------
%
%
\section{Introduction}
In a recent paper Freiberg and Kohl \cite{FK2019} studied the possibility to identify the attractor of a weighted iterated function system with the Martin boundary. To do so, they followed mainly the idea of Denker/Sato \cite{DS2001} and Lau/Ngai/Wang \cite{LauNgai2012, LauWang2015}. They adapted the transition probability of the associated Markov chain and were able to determine the Martin boundary under two conditions. The first condition, called \ref{B1}, is simple and easy to verify. The second condition \ref{B2} is quite harder and it is not obvious, if a fractal fulfills \ref{B2}. Because of this we want to investigate \ref{B2}, determine some properties and make this condition easier to manage.

The outline of this article is as follows: in section \ref{sec:preliminaries} we want to introduce the notation and summarize the work of \cite{FK2019}. We do not need all aspects of this article and thus we only present the most necessary. In section \ref{sec:simplB2} we investigate \ref{B2} and reduce it in the sense that we only have to consider a finite number of words instead of an infinite number as it is in the original work.\\
In section \ref{sec:SG2} we want to apply the ideas of section \ref{sec:simplB2} to the Sierpi\'nski gasket. Thereby we will see that we can represent the Sierpi\'nski gasket as a Martin boundary, but sometimes not all weights of each part can be chosen arbitrarily. After this we want to consider in section \ref{sec:facts} some facts about \ref{B2}, where we introduce further the 3-level Sierpi\'nski gasket and a more general example called $n$-diamond propeller.\\
Last but not least we develop in section \ref{sec:algo} an algorithm to analyze every representation of a fractal. For this we investigate the Sierpi\'nski gasket, the 3-level Sierpi\'nski gasket, the Vicsek fractal, the Pentagasket (with and without hole) and the Hexagasket. As we will see the number of weights which can be chosen arbitrarily vary in a broad way depending on the topology of the fractal.
%
% 
%----------------------------------------------------------------------------------------------------------------------------------------------------
%
%
\section{Preliminaries}
\label{sec:preliminaries}
Let us start with a common method of representing (and generating) fractals: iterated function systems or shorten IFS. For this, we introduce similarities $S_i:D\subseteq \mathbb R^d\to D$ with $|S_i(x)-S_i(y)|=c_i|x-y|$ and $0<c_i<1$. An iterated function system consists of $N$ similarities, where $N$ has to be finite. By Hutchinsons theorem \cite{Hutchinson1981} exists an unique, non--empty compact invariant subset $K\subset \mathbb R^d$ fulfilling 
\begin{equation*}
	K=S(K)\defgl\bigcup_{i=1}^NS_i(K).
\end{equation*}
We want to call $K$ the attractor of the IFS and think of it as the geometric representation of the fractal. Please note, that there are multiple different IFS generating one specific fractal (respectively $K$).\\
Some of those fractals fulfill the so called open set condition, often shorten by OSC. The OSC states that there exists a non--empty bounded open set $\mathcal O\subset\mathbb R^d$ such that $\bigcup_{i=1}^NS_i(\mathcal O)\subset \mathcal O$ with the union disjoint. Since the OSC is a precondition of \cite{FK2019}, we further want to assume that we fulfill the OSC.

Let us examine the fractal in a more accurate way. For this we introduce the alphabet $\mathcal A=\{1,\dots,N\}$ of $N$ letters where each letter represents one similarity. We want further to consider multiple mappings of a set. For this we introduce the word space $\mathcal W$ by
\begin{equation*}
	\mathcal W \defgl \bigcup_{n\geq 1} \mathcal A^n\cup\{\emptyset\},
\end{equation*}
where $\emptyset$ represents the empty word. Let us denote by $\mathcal W^\star$ all $\mathcal A$-valued sequences $w=w_1w_2\dots$.\\
The word space $\mathcal W$ itself consists of words $w$ which can be represented by $w=w_1\dots w_n$ and $w_i\in\mathcal A$. For such a word we define the length $|w| = n$, the parent to be $w^- = w_1\dots w_{n-1}$ and the restriction to the first $m$ letters by $w|_m \defgl w_1\dots w_m$ for $m\leq |w|$. Further we define the product of two words $v, w\in\mathcal W$ with $v= v_1\dots v_m$ and $w=w_1\dots w_n$ by 
\begin{equation*}
	vw = v_1\dots v_mw_1\dots w_n.
\end{equation*}
For the empty word $\emptyset$ we set $|\emptyset| = 0$ and $w\emptyset = \emptyset w = w$ for every $w\in\mathcal W$.\\

Let us now connect the IFS and the word space. For this we define $S_w(E)\defgl S_{w_1}\circ\dots\circ S_{w_n}(E)= S_{w_1}(\cdots(S_{w_n}(E)))$ for $E\subset \mathbb R^d$ and we can think of words $w$ as cells of the fractal (and vice versa). Since some infinite sequences refer to the same point on the fractal, we should identify them as the same. In addition it is convenient to do this also on all finite words.

\begin{defi}[{\cite[Def. 2.1]{FK2019}}]
	\label{defi:sim}
	The words $v,w\in\mathcal W$ are said to be equivalent, noted by $v\sim w$, if and only if $|v|=|w|$, $S_v(K)\cap S_w(K)\neq\emptyset$ and $v^-\neq w^-$. Additionally we say, that $v$ is equivalent to itself, such that $v\sim v$ holds.\\
	For $v,w\in\mathcal W^\star$ with $v=v_1v_2\dots$ and $w=w_1w_2\dots$ we extend this relation, such that $v\sim w$, if and only if there exists a $n_0\in\mathbb{N}$ such that $v|_n\sim w|_n$ holds for all $n\geq n_0$.\\
	Further we want to define the number of equivalent words by $R(w)\defgl\#\{v\in \mathcal W: v\sim w\}$.
\end{defi}

Please notice, that this definition does not guarantee that this relation is indeed an equivalence relation. Luckily this is true for a big number of fractals, in particular for all nested fractals \cite[Prop. 2.9]{FK2019}.\\
In this paper we only want to consider fractals, where $\sim$ forms an equivalence relation (or could be modified such that $\sim$ becomes an equivalence relation).

As a next step we want to introduce a mass distribution $m$ on the fractal respectively on the word space. 
In general it would be possible to define such a mass distribution in a very general way. For our purpose we want to consider relatively simple mass distributions and use the self--similarity of our fractal. We define the mass distribution $m$ for all $a\in\mathcal A$ in such a way that $m(a)\in(0,1)$ and $\sum_{a\in\mathcal A}m(a)=1$ holds. Further we set $m(\emptyset) = 1$. Now we use the self--similarity: for words $w\in\mathcal W$ with $w=w_1\dots w_n$ we define $m(w) \defgl m(w_1)\dots m(w_n)$.\\
If all those weights $m(i)$ are chosen equal (i.e. $m(i) = \frac1N$) we say that we consider the homogeneous fractal or homogeneous case. Otherwise we consider the inhomogeneous case.\\

We want to introduce a Markov chain on $\mathcal W$. For this we adopt \cite[Def. 3.1]{FK2019} and define the transition probability $p: \mathcal W\times \mathcal W \to [0,1]$ by
\begin{equation*}
	p(v,w)\defgl
	\begin{cases}
		\frac{m(w)}{\sum_{\hat v\sim v}m(\hat v)},&\text{ if } w=\hat vi\text{ with } \hat v\sim v \text{ and } i \in\mathcal A,\\
		0,&\text{ else.}
	\end{cases}
\end{equation*}
Further we introduce the $n$-step transition probability $p_n: \mathcal W\times \mathcal W \to [0,1]$ recursively by
\begin{equation*}
	p_n(v,w)\defgl
	\begin{cases}
		\delta_v(w),&n = 0,\\
		\sum_{u\in\mathcal W}p_{n-1}(v,u)p(u,w),& n\geq 1.  
	\end{cases}
\end{equation*}

Since $p_n(v,w) $ is only positive for a specific $n$ we define the Green function $g:\mathcal W\times\mathcal W\to \mathbb R$ by 
\begin{equation*}
	g(v,w)\defgl\sum_{n= 0}^\infty p_n(v,w)=p_{|w|-|v|}(v,w),\qquad v,w\in\mathcal W.
\end{equation*}

The Martin kernel $k:\mathcal W\times\mathcal W\to \mathbb R$ is then defined as a sort of ``renormalized'' Green function, namely by
\begin{equation*}
	k(v,w)\defgl\frac{g(v,w)}{g(\emptyset, w)}=\frac{g(v,w)}{m(w)}, \qquad v,w\in\mathcal W,
\end{equation*}
where we used in the second identity Theorem 3.7 from\cite{FK2019}, which states that $g(\emptyset,w)=m(w)$ for all $w\in\mathcal W$.\\
Based on the Martin kernel one can define the Martin metric $\rho$, which exact form is for now irrelevant. The Martin space $\overline{\mathcal W}$ is then the $\rho$-completion of $\mathcal W$ and the Martin boundary $\mathcal M$ is the boundary of $\overline{\mathcal W}$:
\begin{equation*}
 \mathcal M\defgl \overline{\mathcal W}\,\backslash\mathcal W.
\end{equation*}
This Martin boundary is in some sense connected to the fractal and the attractor. For the homogeneous case this was first studied by \cite{DS2001, DS2002} and later in \cite{JuLauWang2012, LauNgai2012, LauWang2015}. Under the following conditions this can also be done in the inhomogeneous case.

\begin{assumpt}[\cite{FK2019}]
	We make the following assumptions:
	\begin{description}
	\item[(A)\label{A}] The relation $\sim$ is an equivalence relation
		\item[(B1)\label{B1}] The Martin kernel in the homogeneous case exists.
		\item[(B2)\label{B2}] For all $w\in\mathcal W$ holds either
			\begin{align*}
				m(w)&=m(\tilde w)\quad\forall\; \tilde w\sim w\\
				&\text{or}\\
				w^-&\sim(\tilde w)^-\quad\forall\; \tilde w\sim w.
			\end{align*}
	\end{description}
\end{assumpt}

With this we are able to identify the attractor $K$ with the inhomogeneous Martin boundary:

\begin{theorem}[{\cite[Theorem 5.4  / Cor. 5.5]{FK2019}}]
	\label{theo:MBgleich}
	Under \ref{A}, \ref{B1} and \ref{B2} the inhomogeneous Martin boundary coincides with the homogeneous Martin boundary and
	\begin{equation*}
		K\cong \mathcal W^\star\!\big\slash_{\!\!\sim}\cong \mathcal M_{\hom}= \mathcal M.
	\end{equation*}
	holds.
\end{theorem}

This result is stunning. At the same time it arises the question, when the preconditions are fulfilled. Condition \ref{B1} can be easily checked, for example we can use the results of \cite{LauWang2015}. Condition \ref{B2} is quite harder for this reason and in the following we want to study \ref{B2} in the following intensely.
%
% 
%----------------------------------------------------------------------------------------------------------------------------------------------------
%
%
\section{Simplifying condition \ref{B2}}
\label{sec:simplB2}
In a first step we want to examine how we can express the condition \ref{B2} in an easier way. This should also help to detect iterated function systems where \ref{B2} is in the homogeneous case fulfilled. In other words, we want to reverse the problem and analyze which conditions imply \ref{B2}.\\
For this we formulate our first helpful lemma.

\begin{lem}
	\label{lem:simplB2}
	Let $\bar{v} = uv\in\mathcal W$ with $u, v\in\mathcal W$ and $u\neq\emptyset$. If all $\bar{w}\sim \bar{v}$ can be expressed as $\bar{w} = uw$ and $v\sim w$ fulfills \ref{B2}, then $\bar{v}\sim \bar{w}$ fulfills also \ref{B2}.
\end{lem}

\begin{proof}
	Consider $\bar{v} = uv\in\mathcal W$ with $u, v\in\mathcal W$ and $u\neq\emptyset$. Suppose that we can express all equivalent words $\bar{w}\sim \bar{v}$ by $\bar{w} = uw$. Further $v\sim w$ satisfies \ref{B2}. We want do distinguish two cases, depending on which condition of \ref{B2} is fulfilled by $v\sim w$.\\
	In the first case we want to consider that $v\sim w$ fulfill \ref{B2} by $m(v) = m(w)$. This implies that 
	\begin{equation*}
		m(\bar{v}) = m(uv) = m(u)m(v) = m(u)m(w) = m(uw) =m(\bar{w}) 
	\end{equation*}
	holds. Thus $\bar{v}\sim \bar{w}$ fulfill \ref{B2} by the first condition.\\
	Consider as the second case that $v\sim w$ fulfill \ref{B2} by $v^- \sim w^-$. We claim that $\bar{v}^-\sim \bar{w}^-$ holds. To verify this we consider the definition of the equivalence relation.\\
	Both words $\bar{v}^-$ and $\bar{w}^-$ have same length since $|\bar{v}| = |\bar{w}|$ holds. \\
	By $v\sim w$ follows $v^-\neq w^-$ and since $v^-\sim w^-$ must hold we get 
	\begin{equation}
		S_{v^-}(K)\cap S_{w^-}(K)\neq\emptyset. \label{eq:bew:simplB2}
	\end{equation}
	If we apply $S_u(\cdot)$ to \eqref{eq:bew:simplB2} we obtain $S_{\bar{v}^-}(K)\cap S_{\bar{w}^-}(K)\neq\emptyset$.\\
	As last part it remains that $(\bar{v}^-)^-\neq (\bar{w}^-)^-$ holds. By $v^-\sim w^-$ and $v^-\neq w^-$ follows $(v^-)^-\neq (w^-)^-$. This implies $(uv^-)^-\neq (uw^-)^-$ respectively $(\bar{v}^-)^-\neq (\bar{w}^-)^-$.\\
	Thus we receive that $\bar{v}^-\sim \bar{w}^-$ holds which implies that $\bar{v}\sim \bar{w}$ fulfill \ref{B2} by the second condition and we can complete the proof.
\end{proof}

As a next step we want to distinguish between two cases, one which fulfills \ref{B2} automatically and one where we can find an alternative formulation of \ref{B2}.

\begin{kor}
	\label{kor:simplB2}
	Let $v = v_0v_1\dots v_n\in \mathcal W$. If $v_0\neq w_0$ holds for all $w\sim v$ ($w= w_0w_1\dots w_n$) and $v$ should satisfy \ref{B2}, then:
	\vspace{-\topsep}
	\begin{itemize}
		\item if one $w\sim v$ fulfill $w^-\not\sim v^-$: $v$ and all $w$ must fulfill $m(v) = m(w)$,
		\item if all $w\sim v$ fulfill $w^-\sim v^-$: \ref{B2} is automatically fulfilled.
	\end{itemize}
\end{kor}

\begin{proof}
	By Lemma \ref{lem:simplB2} we only have to consider words which differ already in the first letter. Thus we want to consider $v\in\mathcal W$ and all equivalent words $w\sim v$ fulfilling $w_0\neq v_0$. Further \ref{B2} should hold.\\
	In the first case there exists an equivalent word $w$ with $w^-\not\sim v^-$. Since \ref{B2} must hold, the first condition must be fulfilled and thus
	\begin{equation*}
	 m(w) = m(v)\text{\quad for all }w\sim v.
	\end{equation*}
	In all other cases all equivalent words $w$ fulfill $w^-\sim v^-$. Immediately it follows that \ref{B2} is fulfilled by the second condition.
\end{proof}

We can combine the previous lemma and corollary into one single statement, which specify when \ref{B2} is fulfilled.

\begin{prop}
	\label{prop:simplB2}
	Consider $v=v_0\dots v_n$ and $w= w_0\dots w_n$. If all relations $v\sim w$ with $v_0\neq w_0$ and $v^-\not\sim w^-$ fulfill $m(v) = m(\tilde w)$ for all $\tilde w\sim v$, then \ref{B2} is satisfied.
\end{prop}

\begin{proof}
	Let $v,w\in\mathcal W$. Lemma \ref{lem:simplB2} implies that we only have to consider relations where the first letter differs. If further the parents $v^-, w^-$ are equivalent for all $w\sim v$ we can apply the second case of Corollary \ref{kor:simplB2} and \ref{B2} is fulfilled.\\
	If there exists one equivalent word $w\sim v$ which parent is not equivalent to $v^-$, but at the same time $m(\tilde w) = m(v)$ holds for all $\tilde w\sim v$, \ref{B2} is satisfied by its second condition.
\end{proof}

We want to use this proposition extensive in the following, since it allows us to check only a small amount of words in $\mathcal W$. In contrast to this we would have to check all words in $\mathcal W$, if we want to verify \ref{B2} in a direct way. This is indeed very problematic, since we would have to check infinite many words. In other words Proposition \ref{prop:simplB2} allows us to check only finite many words because of the self--similarity. As we will see this amount of words is very low. For example we only have to check 3 relations on the Sierpi\'nski gasket, which we will see in the next section.

%
% 
%----------------------------------------------------------------------------------------------------------------------------------------------------
%
%
\section{Sierpi\'nski gasket as the simplest example}
\label{sec:SG2}
We want to apply and understand the results from section \ref{sec:simplB2}. For this, we consider one of the simplest fractals one can think of: the Sierpi\'nski gasket. %
To be more precise, we want to consider the Sierpi\'nski gasket in the plane, shorten by SG. The SG is generated through three smaller copies of a triangle, which are arranged such that they form again the starting triangle with a hole. This is done infinite times and the resulting figure is called Sierpi\'nski gasket. Figure \ref{fig:generatingSG2} should help to understand this procedure. %
\begin{figure}
	\centering
	\includegraphics[page=1]{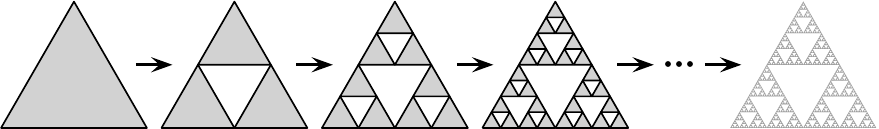}
	\caption{Generating the Sierpi\'nski gasket in a descriptive way}
	\label{fig:generatingSG2}
\end{figure}
Of course one could also specify the three similarities in a proper mathematical way, but this would only distract ourselves from the problem, since it is more a topological problem. Instead we want to analyze in which way each copy can be arranged. So let us consider on single small copy. This copy can be rotated by 0, $\frac23\pi$ or $-\frac23\pi$ and may be flipped over. In combination we get six possible ways to map the big triangle onto the small copy. Those six ways are also illustrated in figure \ref{fig:SG2-6ways}.\\%
\begin{figure}
	\centering
	\subcaptionbox{The standard mapping with no rotation and no flipping\label{fig:SG2-6ways-way1}}[0.30\textwidth][t]{\centering\includegraphics[page=2]{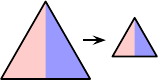}}%
	\hfill%
	\subcaptionbox{No rotation, but flipping.\label{fig:SG2-6ways-way2}}[0.30\textwidth][t]{\centering\includegraphics[page=3]{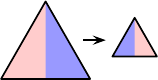}}%
	\hfill%
	\subcaptionbox{Rotate the triangle by $-\frac23\pi$ without flipping.\label{fig:SG2-6ways-way3}}[0.30\textwidth][t]{\centering\includegraphics[page=4]{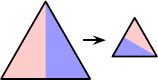}}%
	\\%
	\subcaptionbox{Rotate the triangle by $-\frac23\pi$ and flip it over.\label{fig:SG2-6ways-way4}}[0.30\textwidth]{\centering\includegraphics[page=5]{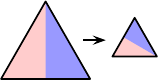}}%
	\hfill%
	\subcaptionbox{Rotate the triangle by $\frac23\pi$ without flipping.\label{fig:SG2-6ways-way5}}[0.30\textwidth]{\centering\includegraphics[page=6]{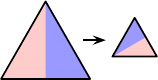}}%
	\hfill%
	\subcaptionbox{Rotate the triangle by $\frac23\pi$, further flip it.\label{fig:SG2-6ways-way6}}[0.30\textwidth]{\centering\includegraphics[page=7]{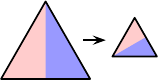}}%
	\caption{6 possible mappings for the smaller copies at the Sierpi\'nski gasket. The coloring should help to understand the orientation and the flipping. }
	\label{fig:SG2-6ways}
\end{figure}%
We can apply one of the six mappings independently on all three copies, thus we get in total $6^3 = 216$ possible IFS. This number initially appears small. But if we consider other fractals the number of possible IFS increases rapidly. Because of this, we should consider all IFS in a general way.\\
\begin{figure}
	\centering
	\includegraphics[page=8]{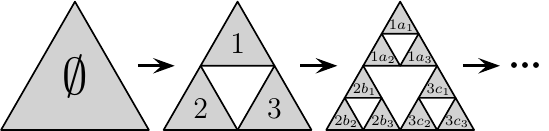}
	\caption{The general word space of the Sierpi\'nski gasket. It must hold, that $a_i$ (resp. $b_i$ and $c_i$) are pairwise different and $a_i, b_i, c_i\in\{1,2,3\}$.}
	\label{fig:SG2-generalWordspace}
\end{figure}%
For the Sierpi\'nski gasket we want to write down the general word space. This can be seen in Figure \ref{fig:SG2-generalWordspace} for words of length $2$. We choose the coding in such a way that the first letter is fixed and thus the rotation and flipping does not matter for the first letter. For the second letter they come into account and the second letter has to variable. We can denote the children of $1$ by $1a_i$ and similar the children of $2$ (resp. $3$) by $b_i$ ($c_i$). It must hold, that $a_1, a_2$, $a_3\in\{1,2,3\}$ and are pairwise distract. The same must hold for $b_i$ respectively $c_i$. By the definition of the equivalence relation it follows, that $1a_2\sim 2b_1$, $1a_3\sim 3c_1$ and $2b_3\sim 3c_2$ must hold. By Proposition \ref{prop:simplB2} it is sufficient, to consider only those three relations. Thus, if \ref{B2} should be fulfilled, the equations
\begin{equation}
	\begin{gathered}
		m(1a_2) = m(2b_1),\\
		m(1a_3) = m(3c_1),\\
		m(2b_3) = m(3c_2)
	\end{gathered}%
	\label{eq:GeneralSG2}
\end{equation}
must hold. Remember, that $m$ is multiplicative, so $m(1a_2) = m(1)m(a_2)$. Since $m$ is a mass distribution, we have to add
\begin{equation}
	\begin{gathered}
		m(1)+m(2)+m(3) = 1,\\
		m(a)> 0, a\in\mathcal A
	\end{gathered}%
	\label{eq:SG2MassDistribution}
\end{equation}
to the equations in \eqref{eq:GeneralSG2}. In total we get four algebraic equations in three positive variables $m(1)$, $m(2)$ and $m(3)$. The only problem are the values of $a_1,\dots, c_3$, which differ with each IFS. For this reason we want consider in the following three different examples and we will get a deeper insight in the dependency of choosing small mappings fulfilling \eqref{eq:GeneralSG2} and \eqref{eq:SG2MassDistribution}.
\begin{bsp}
	\label{bsp:SG2-0fp}
	Consider the IFS with $a_1 = 1, a_2 = 2, a_3= 3, b_1=3, b_2 = 1, b_3=2, c_1=2, c_2=3, c_3=1$. Figure \ref{fig:SG2-0fp} shows this IFS in a graphical way, moreover the word space is indicated.\\
	\begin{figure}%[h]
		\centering
		\begin{subfigure}[b]{0.99\textwidth}
			\centering
			\includegraphics[page=9]{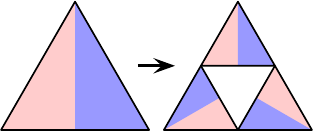}
			\caption{Iterated function system}
		\end{subfigure}
		\\[1ex]
		\begin{subfigure}[b]{0.99\textwidth}
			\centering
			\includegraphics[page=10]{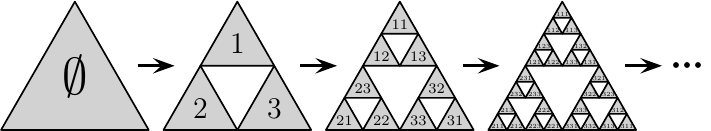}
			\caption{Word space $\mathcal W$}
		\end{subfigure}
		\caption{An IFS where no weights can be chosen. The coloring should help to understand the orientation of the contractions. It must hold, that $m(1) = m(2) = m(3)$ and thus $m(a) = \frac13$ for $a\in\mathcal A$.}
		\label{fig:SG2-0fp}
	\end{figure}%
	The equations \eqref{eq:GeneralSG2} and \eqref{eq:SG2MassDistribution} turn into
	\begin{equation*}
	    \begin{gathered}
		m(1)m(1) = m(2)m(3),\\
		m(1)m(3) = m(3)m(2),\\
		m(2)m(2) = m(3)m(2),\\
		m(1) + m(2) + m(3) = 1.
	    \end{gathered}
	\end{equation*}
	Those equations can easily be solved (unlike the general equations in \eqref{eq:GeneralSG2}). We get, that $m(1) = m(2) = m(3) = \frac13$ must hold. We conclude therefore, that this IFS only fulfills \ref{B2} in the homogeneous case, which is (for us) relatively uninteresting.
\end{bsp}
The next example is much more interesting, since we get a first weighted example of the Sierpi\'nski gasket fulfilling \ref{B2}.
\begin{bsp}
	\label{bsp:SG2-1fp}
	Let $a_1 = 1, a_2 = 2, a_3= 3, b_1=1, b_2 = 2, b_3=3, c_1=2, c_2=1, c_3=3$. This iterated function system and the word space are illustrated in Figure \ref{fig:SG2-1fp}.\\
	\begin{figure}
		\centering
		\begin{subfigure}[b]{0.99\textwidth}
			\centering
			\includegraphics[page=11]{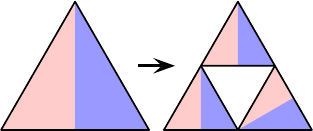}
			\caption{Iterated function system}
		\end{subfigure}
		\\[1ex]
		\begin{subfigure}[b]{0.99\textwidth}
			\centering  
			\includegraphics[page=12]{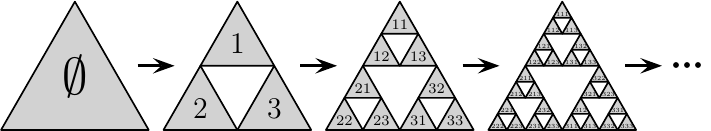}
			\caption{Word space $\mathcal W$}
		\end{subfigure}
		\caption{One weight can be chosen. In this case, we can choose $m(3) \in(0,1)$. It follows, that $m(1) = m(2) = \frac{1-m(3)}{2}$ must hold.}
		\label{fig:SG2-1fp}
	\end{figure}%
	The corresponding equations, which need to be fulfilled in order to fulfill \ref{B2}, are then 
	\begin{equation*}
	      \begin{gathered}
		  m(1)m(2) = m(2)m(1),\\
		  m(1)m(3) = m(3)m(2),\\
		  m(2)m(3) = m(3)m(1),\\
		  m(1) + m(2) + m(3) = 1.
	      \end{gathered}
	\end{equation*}
	Immediately one can see that $m(1) = m(2)$ must hold. At the same time the equations make no restrictions on $m(3)$, except from $2m(1) + m(3) = 1$. This implies that we either can choose $m(3)\in(0,1)$ with $m(1) = m(2) = \frac{1-m(1)}2$ or we can choose $m(1)\in (0,\frac12)$ with $m(3) = 1 - 2 m(1)$ (and of course $m(2) = m(1)$).\\
	In both cases we can choose the weights such that they are inhomogeneous and at the same time \ref{B2} is fulfilled. Since we can choose one weight, we want to call this IFS a case with one free weight or one free parameter. 
\end{bsp}
The previous examples show that it depends on the IFS, if \ref{B2} can be fulfilled in the inhomogeneous case. The question arises whether \ref{B2} is fulfilled and we can choose more than one weight arbitrarily. The next example should help to clear this question.
\begin{bsp}
	\label{bsp:SG2-2fp}
	Consider the original Sierpi\'nski gasket without rotations or flippings (see Figure \ref{fig:SG2-2fp}). In this case we have $a_1 = 1, a_2 = 2, a_3= 3, b_1=1, b_2 = 2, b_3=3, c_1=1, c_2=2$ and $c_3=3$ with
	\begin{equation*}
		\begin{gathered}
			m(1)m(2) = m(2)m(1),\\
			m(1)m(3) = m(3)m(1),\\
			m(2)m(3) = m(3)m(2),\\
			m(1) + m(2) + m(3) = 1.
		\end{gathered}
	\end{equation*}
	\begin{figure}
		\centering
		\begin{subfigure}[b]{0.9\textwidth}
			\centering
			\includegraphics[page=13]{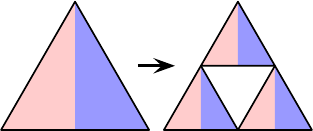}
			\caption{Iterated function system}
		\end{subfigure}
		\\[1ex]
		\begin{subfigure}[b]{0.99\textwidth}
		    \centering
		    \includegraphics[page=14]{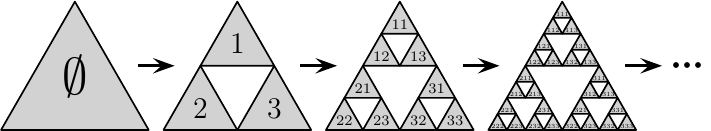}
		    \caption{Word space $\mathcal W$}
		\end{subfigure}
		\caption{Two weights can be chosen. For example, if we choose $m(1)\in(0,1)$, we can further choose $m(2)\in (0,1-m(1))$. This implies $m(3) = 1-m(1)-m(2)$.}
		\label{fig:SG2-2fp}
	\end{figure}%
	The first three equations are automatically fulfilled and only $m(1) + m(2) + m(3) = 1$ remains. Thus we can choose two weights and the third is determined by this equation. For example we can choose first $m(1)\in(0,1)$. Then we choose $m(2)\in(0,1-m(1))$, where the upper bound of $1-m(1)$ comes from the fact, that $m(1) + m(2) < 1$ must hold because of $m(3)>0$. We can proceed in the same way, if we want to choose other weights.\\
	In total we have two free parameters on this particular IFS.
\end{bsp}
Those three examples show that it purely depends on the topology of each fractal. In other words, we can consider a certain fractal and may look at all iterated function systems which generate this fractal. Depending on the topology we get a different number of free parameters. If we do this for the remaining 213 IFS of the Sierpi\'nski gasket we get that there are in total 194 IFS with zero free parameters, 21 IFS with one free parameter and exactly one (!) IFS with two free parameters, the one from Example \ref{bsp:SG2-2fp}.\\%
Because of this, we want to establish in section \ref{sec:algo} an algorithm which analyzes all possible IFS to a certain fractal and returns the number of free parameters. But before we do this, let us first take a closer look at some facts about \ref{B2}.
%
% 
%----------------------------------------------------------------------------------------------------------------------------------------------------
%
%
\section{Some general facts}
\label{sec:facts}
We want to collect in this section some facts about (inhomogeneous) fractals and the condition \ref{B2}. As a first topic we want to answer the question if there is a minimal inhomogeneous fractal which fulfills \ref{B2}. Later we want to set up a minimal example which won't fulfill \ref{B2}.\\
Those minimal examples should fulfill some conditions such that they are actually real minimal examples:
\begin{itemize}
	\item We want the number of copies resp. the length of the alphabet $\mathcal A$ to be minimal.
	\item The open set condition should hold. This condition seems may be not relevant for finding a minimal example, since we could consider the so called Haka Tree (for more information see \cite[Example 1.2.9]{Kigami2001}. But since we want to apply the results of \cite{FK2019} we need the OSC. 
	\item All those examples should be indeed fractals (see for example \cite[Introduction]{FalconerFG} for a definition of a fractal through a list of properties). Of course, one could define the interval $[0,1]$ as a self--similar fractal with two (or more copies), but it is obvious that this is only an artificial fractal.
	\item Those examples should contain at least a word $w$ with $R(w) > 1$. If we would consider a fractal where all words $w$ fulfill $R(w) = 1$, then \ref{B2} would be fulfilled automatically. This would make the minimal example meaningless.
	\item The mass distribution should be inhomogeneous, otherwise \ref{B2} is also automatically fulfilled.
\end{itemize}
In other words, our minimal example should be an inhomogeneous fractal, minimal in the sense of the number of copies and should contain a word $w$ with $R(w) > 1$.
\begin{bsp}[Minimal inhomogeneous fractal fulfilling \ref{B2}]
\label{bsp:minB2}
    \begin{figure}
    \centering
    \subcaptionbox{The IFS generating the von Koch curve.\label{fig:KochKurve1-IFS}}[\textwidth]{
		\includegraphics[page=15]{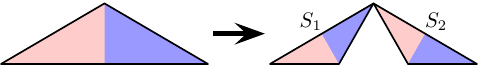}
    }
    \\[4ex]
    \subcaptionbox{The word space up to length $2$ and the attractor of the von Koch curve.\label{fig:KochKurve1-wordspace}}[\textwidth]{
		\includegraphics[page=16]{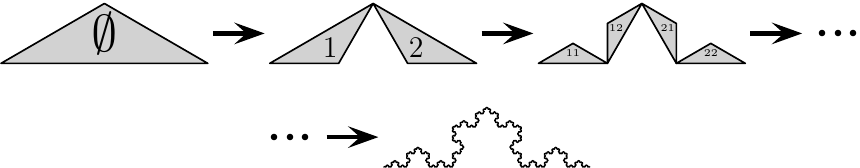}
    }
    \caption{The von Koch curve generated with two similarities and \ref{B2} is fulfilled with an inhomogeneous mass distribution. \ref{A} contains the IFS in a graphical way and the coloring should help to understand the orientation of the small copies.\\
    (B) contains the associated word space (up to length $2$) and the attractor.}
    \label{fig:KochKurve1}
    \end{figure}
    Our minimal example needs at least two small copies, otherwise we are unable to define an IFS. In fact this is possible with the von Koch curve, which is named after the Swedish mathematician Helge von Koch (1870 -- 1924) who introduced this fractal in 1904 \cite{Koch1904}.
    
    The typical IFS which generates the von Koch curve uses normally four small copies and arranges them in such a way that the small copies form the typical tent--form. Besides this there is another IFS which has the von Koch curve as attractor. This IFS consists of the two mappings
    \begin{equation*}
     \begin{aligned}
S_1: \mathbb R^2\to\mathbb R^2, y\mapsto &~\frac1{\sqrt{3}}
                 \begin{pmatrix}
					-\cos\left( \frac{7 \pi}{6}\right)   &   -\sin\left( \frac{7 \pi}{6}\right)\\
					-\sin\left( \frac{7 \pi}{6}\right)   &   \cos\left( \frac{7 \pi}{6}\right)
                 \end{pmatrix}
                 y =\\
&~ \frac1{\sqrt{3}}
                 \begin{pmatrix}
                    \frac{\sqrt{3}}{2}   &   \frac{1}{2}\\
                    \frac{1}{2}   &   - \frac{\sqrt{3}}{2}
                 \end{pmatrix}
                 y,
                \\
S_2: \mathbb R^2\to\mathbb R^2, y\mapsto &~\frac1{\sqrt{3}}
                 \begin{pmatrix}
					-\cos\left( \frac{5 \pi}{6}\right)   &   -\sin\left( \frac{5 \pi}{6}\right)\\
					-\sin\left( \frac{5 \pi}{6}\right)   &   \cos\left( \frac{5 \pi}{6}\right)
                 \end{pmatrix}
                 y + 
\begin{pmatrix}
                   \frac{1}{2}\\
                   \frac{\sqrt{3}}{6}
                    \end{pmatrix}=\\
&~ \frac1{\sqrt{3}}
                 \begin{pmatrix}
                    \frac{\sqrt{3}}{2}   &   - \frac{1}{2}\\
                    - \frac{1}{2}   &   - \frac{\sqrt{3}}{2}
                 \end{pmatrix}
                 y + 
\begin{pmatrix}
                   \frac{1}{2}\\
                   \frac{\sqrt{3}}{6}
                    \end{pmatrix}.
     \end{aligned}
    \end{equation*}
    The two mappings are also presented in Figure \ref{fig:KochKurve1-IFS} in a graphical way, which might be much easier to understand. The coloring of the mappings should help to understand how the small copies are arranged. Figure \ref{fig:KochKurve1} contains further a second figure. The Figure \ref{fig:KochKurve1-wordspace} shows the word space up to length $2$ and further the attractor of the IFS.
    
    The alphabet consists obviously of $\mathcal A=\{1,2\}$ and the word space of all words $w=w_1w_2\dots w_n$ with $w_i\in\mathcal A$. We apply Proposition \ref{prop:simplB2} and hence the equation
    \begin{equation}
     m(12) = m(21)\label{eq:bsp:KochKurve1-eq1}
    \end{equation}
    has to be fulfilled. Since $m$ is multiplicative, it follows that \eqref{eq:bsp:KochKurve1-eq1} is always fulfilled. So, the only restriction is the fact that $m$ has to be a mass distribution with 
    \begin{equation*}
    m(1) + m(2) = 1.
    \end{equation*}
    Therefore it follows that we can choose $m(1)\in(0,1)$ and $m(2)=1-m(1)$. Of course we could also choose $m(2)\in(0,1)$ and determine $m(1)$.
    
    Thus we found an inhomogeneous fractal which fulfills \ref{B2}. The attractor is also a ``real'' fractal, since the Hausdorff dimension equals $\frac{\ln 4}{\ln 3}\approx \numprint{1.262}$. We also note that this fractal is p.c.f..                   
\end{bsp}

As a logical consequence we can ask for a minimal example which won't fulfill \ref{B2}.

\begin{bsp}[Minimal inhomogeneous fractal not fulfilling \ref{B2}]
    \begin{figure}
    \centering
    \subcaptionbox{The IFS.\label{fig:KochKurve2-IFS}}[\textwidth]{
		\includegraphics[page=17]{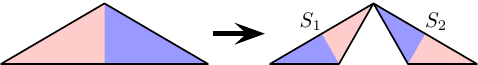}
    }
    \\[4ex]
    \subcaptionbox{The word space and the attractor.\label{fig:KochKurve2-wordspace}}[\textwidth]{
		\includegraphics[page=18]{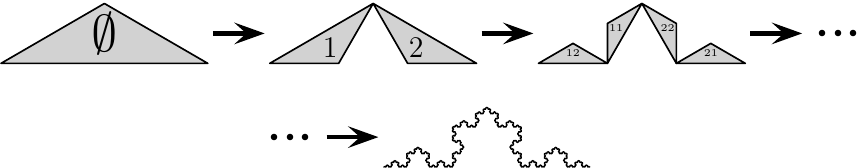}
    }
    \caption{An IFS (and its word space) generating the von Koch curve, but which fulfill \ref{B2} only in the homogeneous case.}
    \label{fig:KochKurve2}
    \end{figure}
    For this minimal counter--example we use again the von Koch curve, which we already introduced in the previous Example \ref{bsp:minB2}. For the counter--example we arrange the small copies in a different way. We choose the mappings as 
    \begin{equation*}
        \begin{aligned}
        S_1: \mathbb R^2\to\mathbb R^2, y\mapsto &~- \frac1{\sqrt{3}}
            \begin{pmatrix}
                \cos\left( \frac{\pi}{6}\right)   &   -\sin\left( \frac{\pi}{6}\right)\\
                \sin\left( \frac{\pi}{6}\right)   &   \cos\left( \frac{\pi}{6}\right)
            \end{pmatrix}
            y + 
            \begin{pmatrix}
                \frac{1}{2}\\
                \frac{\sqrt{3}}{6}
            \end{pmatrix}=\\
        &~ - \frac1{\sqrt{3}}
            \begin{pmatrix}
                \frac{\sqrt{3}}{2}   &   - \frac{1}{2}\\
                \frac{1}{2}   &   \frac{\sqrt{3}}{2}
            \end{pmatrix}
            y + 
            \begin{pmatrix},
                \frac{1}{2}\\
                \frac{\sqrt{3}}{6}
            \end{pmatrix}\\
        S_2: \mathbb R^2\to\mathbb R^2, y\mapsto &~- \frac1{\sqrt{3}}
            \begin{pmatrix}
                \cos\left( - \frac{\pi}{6}\right)   &   -\sin\left( - \frac{\pi}{6}\right)\\
                \sin\left( - \frac{\pi}{6}\right)   &   \cos\left( - \frac{\pi}{6}\right)
            \end{pmatrix}
            y + 
            \begin{pmatrix}
                1\\
                0
            \end{pmatrix}=\\
        &~ - \frac1{\sqrt{3}}
            \begin{pmatrix}
                \frac{\sqrt{3}}{2}   &   \frac{1}{2}\\
                - \frac{1}{2}   &   \frac{\sqrt{3}}{2}
            \end{pmatrix}
            y +
            \begin{pmatrix}
                1\\
                0
            \end{pmatrix}
        \end{aligned}
    \end{equation*}
    which is also illustrated in Figure \ref{fig:KochKurve2-IFS}. The second Figure \ref{fig:KochKurve2-wordspace} contains the corresponding word space and again the attractor. The intention of this IFS is the fact that the cells ``11'' and ``22'' touch each other and are therefore equivalent.
    
    By Proposition \ref{prop:simplB2} the equation
    \begin{equation*}
        m(11)=m(22)
    \end{equation*}
    has to be fulfilled. For $m(1)>0$ and $m(2)>0$ this is only fulfilled in the case of $m(1) = m(2)$. This means that an inhomogeneous mass distribution on this IFS can not fulfill \ref{B2}.
    
    This example is a valid minimal example using the same arguments as in Example \ref{bsp:minB2} and is also p.c.f..
    
\end{bsp}
In total we get two examples, which fulfill our criteria of a minimal example. Thereby, one example fulfills \ref{B2} and the other example does not fulfill \ref{B2}.

As a next step we want to consider so called nested fractals, which are characterized as follows:
\begin{defi}[\cite{Hambly2000}]
	\label{defi:nested}
	We want to denote by $F_0'\defgl\{q_i: S_i(q_i)=q_i\}$ the set of all fixed points of the similarities $S_i$. Further we want do define the set of all essential fixed points $F_0$ by $F_0\defgl\{x\in F_0': \exists i, j\in\mathcal A,  y \in F_0', x\neq y\text{ st. } S_i(x)=S_j(y)\}$.
	A fractal $K$ is then called nested, if it satisfies:
	\begin{enumerate}
		\item Connectivity: For any $1$-cells $C$ and $C'$, there is a sequence $\{C_i: i=0,\dots, n\}$ of $1$-cells such that $C_0=C, C_n = C'$ and $C_{i-1}\cap C_i\neq \emptyset, i=1,\dots, n$.
		\item Symmetry: If $x,y\in F_0$, then reflection in the hyperplane $H_{xy}=\{Z:|z-x|=|z-y|\}$ maps $S^n(F_0)$ to itself.
		\item Nesting: If $v,w\in\mathcal W$ with $v\neq w$, then 
		\begin{equation*}
			S_v(K)\cap S_w(K) = S_v(F_0)\cap S_w(F_0)
		\end{equation*}
		\item Open set condition OSC: There is a non-empty, bounded, open set $\mathcal O$ such that the $S_i(\mathcal O)$ are disjoint and $\bigcup_{i=1}^NS_i(\mathcal O)\subseteq \mathcal O$.
	\end{enumerate}
\end{defi}
Those nested fractals form a class of fractals with some nice properties. One of those properties is the following: if the fractal is nested, then the equivalence relation is indeed equivalent \cite[Prop. 2.9]{FK2019}. Since we only want to consider fractals with such an equivalence relation, this is a nice pre--condition.\\
Unfortunately, the property ``nested'' does not imply that \ref{B2} is fulfilled in the inhomogeneous case, which shows the next example.
\begin{bsp}
	\label{bsp:SG3-nested}
	Let us consider the so called 3-level Sierpi\'nski gasket, shorten by SG$_3$. This is a modification of the normal Sierpi\'nski gasket and consists of six copies with contraction ratios $\frac13$. Those are arranged in such a way that they form again a triangle. In Figure \ref{fig:SG3-nested} this is also visualized. Again, we could flip and rotate the smaller copies, but for now we want to consider the case with no flipping or rotating. The associated word space is included in Figure \ref{fig:SG3-nested}. It holds that in this case the SG$_3$ is nested, where $F_0$ consists of the edges of the starting triangle.
	\begin{figure}
		\centering
		\includegraphics[page=19]{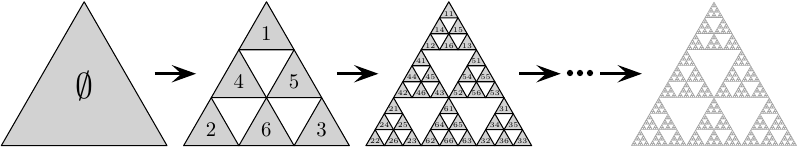}
		\caption{The 3-level Sierpi\'nski gasket, which is generated in a same way as the Sierpi\'nski gasket, but with 6 smaller copies. This SG$_3$ has no rotations or flippings.}
		\label{fig:SG3-nested}
	\end{figure}
	By Proposition \ref{prop:simplB2} we only have to consider the relations $12\sim 41$, $13\sim 51$, $21\sim42$, $23\sim62$, $31\sim53$, $32\sim63$, $43\sim52\sim61$. We can set up all needed equations and become
	\begin{equation*}
		\begin{split}
			& m(12)=m(41)\\
			& m(13)=m(51)\\
			& m(21)=m(42)\\
			& m(23)=m(62)\\
			& m(31)=m(53)\\
			& m(32)=m(63)\\
			& m(43)=m(52)=m(61)\\
			& m(1) + m(2) + m(3) + m(4) + m(5) + m(6) = 1
		\end{split}
	\end{equation*}
	We can apply the fact that $m(ab) = m(a)m(b)$ holds, shorten the equations and receive
	\begin{equation*}
		\begin{split}
			& m(2)=m(4)\\
			& m(3)=m(5)\\
			& m(1)=m(4)\\
			& m(3)=m(6)\\
			& m(1)=m(5)\\
			& m(2)=m(6)\\
			& m(43)=m(52)=m(61)\\
			& m(1) + m(2) + m(3) + m(4) + m(5) + m(6) = 1.
		\end{split}
	\end{equation*}
	As we can see, it must hold that $m(a) = \frac16$ for $a\in\mathcal A$, which implies that we are not able to choose an inhomogeneous mass distribution and at the same time \ref{B2} is fulfilled.\\
	Thus this is a suitable counter example for the fact, that the property ``nested'' cannot imply the fact that \ref{B2} is fulfilled in the inhomogeneous case.
\end{bsp}
In the following we want to show, that \ref{B2} is not to restrictive. In particular we want to consider inhomogeneous fractals, which have words $\hat w$ with $R(\hat w)>2$. We split this into two examples, since the first example is a special example with $R(\hat w) = 3$ and the second example is more general with $R(\hat w) = n, n\geq 4$.
\begin{bsp}
	\label{bsp:SG3-1fp}
	For an example with $R(\hat w)=3$ we can use the 3-level Sierpi\'nski gasket, which we already introduced in the previous Example \ref{bsp:SG3-nested}. In contrast to this example we want to rotate some copies, but won't flip. If we rotate $S_2$ and $S_3$ by $-\frac23\pi$ and $S_6$ by $\frac23\pi$ we get a different word space, which is illustrated in Figure \ref{fig:SG3-1fp}. %
	\begin{figure}
		\centering
		\begin{subfigure}[b]{0.9\textwidth}
			\centering
			\includegraphics[page=20]{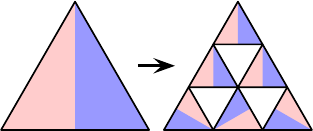}
			\caption{Iterated function system}
		\end{subfigure}
		\\[1ex]
		\begin{subfigure}[b]{0.99\textwidth}
			\centering
			\includegraphics[page=21]{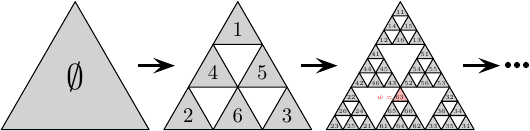}
			\caption{Word space $\mathcal W$}
		\end{subfigure}
		\caption{The Sierpi\'nski gasket with 6 copies and rotations, which has a cell $\hat w=63$ with $R(\hat w)=3$ and $\hat w$ fulfills \ref{B2}.}
		\label{fig:SG3-1fp}
	\end{figure}
	As a first observation we note, that $\sim$ forms still an equivalence relation. Further we can choose the mass of the cell $1$ with $m(1)\in(0,1)$. For all other cells it holds, that $m(a) = \frac{1-m(1)}5$ for $a\in\{2,\dots, 6\}$.\\
	It then holds, that \ref{B2} is fulfilled for all $w\in\mathcal W$. For this, we can analyze words of length $2$, which will not fulfill the property \ref{B2} by $w^-\sim (\tilde w)^-$. Thus they must fulfill $m(w) = m(\tilde w)$ for all $\tilde w\sim w$, which shows the following list:
	\begin{align*}
		m(12)&=m(41)\\
		m(42)&=m(22)\\
		m(21)&=m(61)\\
		m(62)&=m(33)\\
		m(32)&=m(53)\\
		m(51)&=m(13)\\
		m(43)&=m(52)=m(63)
	\end{align*}
	For a word of length greater 2 we can differ between two cases. Either $w$ fulfills $w^-\sim(\tilde w)^-$ for all $\tilde w\sim w$. In this case \ref{B2} is fulfilled. In the other case $w$ can be expressed as $w=uab$ and $\tilde w=ucd$ with $u\in\mathcal W$ and $a,b, c, d\in\mathcal W$. It then holds, that we can apply the results for words of length $2$, where $m(ab) = m(cd)$ holds. Thus $m(w) = m(uab) = m(ucd) = m(\tilde w)$ follows and \ref{B2} is fulfilled for all $w\in\mathcal W$.\\
	For now it could be possible, that this example can be modified in such a way, that we can choose the mass of two cells arbitrarily. As we will later see in section \ref{sec:algo}, there is no way of generating the SG$_3$ with six copies and two or more free parameters.
\end{bsp}
\begin{bsp}[$n$-diamond propeller]
	The previous Example \ref{bsp:SG3-1fp} already gives a good idea, how we could construct an example (in $\mathbb R^2$) with a cell $\hat w$ and $R(\hat w)=n$. Basically, we just need $n$ touching copies of the fractal in one single point $p$.  We can achieve this, if we choose our basic form to be a diamond and we want to call this example the \emph{$n$-diamond propeller}.\\
	Since we do not want overlapping copies the angle in one corner of the diamond must be sufficiently small. %
	% Let us denote this angle by $\alpha$.\\
	For now, we want to denote this angle by $\tilde\alpha$. We notate it with a tilde, since we have to modify this angle and denote later the final angle by $\alpha$.\\
	Let us fix $n\geq 4$ as the number of touching copies in $p$. If we choose $\tilde\alpha\leq \frac{\pi}{n}$, we can arrange the copies in such a way, that they only touch in $p$. As it turns out, it is even more practically, if we choose 
	\begin{equation*} 
		\alpha_0\defgl  \pi\left(4\left\lceil\frac{n}4\right\rceil\right)^{-1} \leq \tilde\alpha,
	\end{equation*}
	where $\lceil\cdot\rceil$ is the ceiling function.\\
	This allows us, that four copies form a ``+'' in $p$. We can then choose the contraction ratio $c_0$ of all similarities as 
	\begin{equation*}
		c_0\defgl\frac12\tan\left(\frac{\alpha_0}2\right).
	\end{equation*}
	This implies, that the smaller copies just fit horizontal in the original diamond.\\
	In the next step we add at the top and the bottom of the diamond two additional copies, which guarantees us, that the shape of the OSC--set $\mathcal O$ is (at least) diamond--like.\\
	In the last construction step we want to connect those outer cells with the inner cells. We want to avoid copies with improper rotations, instead they should form a proper line. For this, we have to choose the contraction ratio smaller and it should be a multiple of four, such that we set the new contraction ratio $c$ as 
	\begin{equation*} 
		c\defgl \left(4\left\lceil\frac1{4c_0}\right\rceil\right)^{-1}.
	\end{equation*}
	As we still want to hold on the property that the smaller $n$ copies fit exactly in the diamond, we also have to adjust $\alpha$ and set 
	\begin{equation*} 
		\alpha\defgl 2\arctan(2c). 
	\end{equation*}
	This also implies, that we have to modify our OSC--set $\mathcal O$ slightly. We choose the union of the diamond and a circle around $p$ with radius $c$ as new set $\mathcal O$. The different steps of our construction can also be seen in Figure \ref{fig:diamond1}.\\
	\begin{figure}
		\centering
		\includegraphics[page=22]{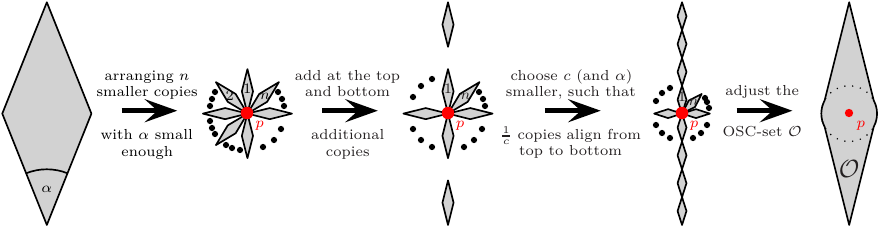}
		\caption{Idea of the construction of a fractal with a cell $w$ and $R(w) = n$.}
		\label{fig:diamond1}
	\end{figure}%
	Let us now consider the word space. For this, we denote the cells in line from top to bottom by 
	\begin{equation*}
		1, \dots, n_0
	\end{equation*}
	with $n_0\defgl \frac1c$. The remaining cells arranged in a circle around $p$ should be denoted by 
	\begin{equation*}
		n_0+1,\dots, N
	\end{equation*}
	where $N = n_0-2+n$ holds. We want to orientate the copies from top to bottom in such a way, that two neighboring cells $u$ and $v$ touch in $u1$ and $v1$ respectively $un_0$ and $vn_0$, except the cell $n_0$. The cell $n_0$ is not rotated and it should hold, that $n_01$ intersects with the cell $(n_0-1)n_0$.\\
	Since $n_0$ is a multiple of four, it holds, that the cells $\frac{n_0}2$ and $(\frac{n_0}2+1)$ meet exactly in $p$. Moreover their orientation is in the way, that they intersect in $(\frac{n_0}2)1$ and $(\frac{n_0}2+1)1$. The remaining copies are accumulated around $p$. For such a copy $w$ it should hold that $w1$ intersects with the other cells. Figure \ref{fig:diamond2} should help to understand this word space.\\
	\begin{figure}
		\centering
		\includegraphics[page=23]{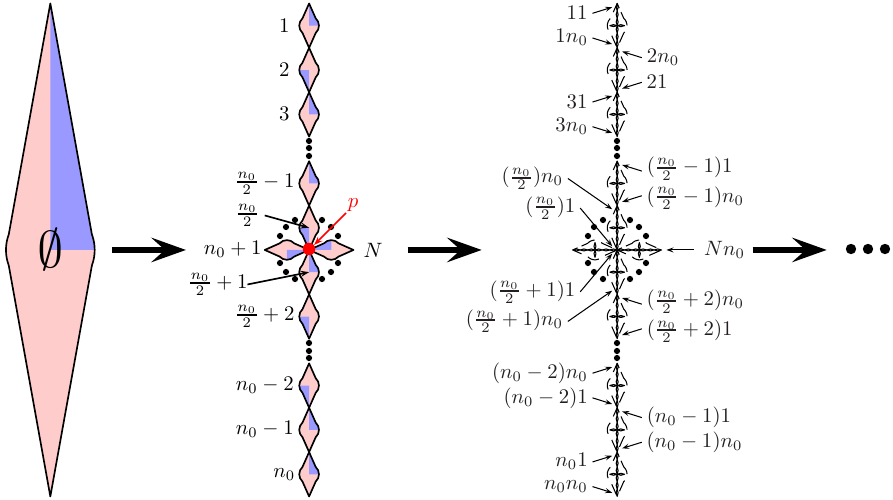}
		\caption{The fractal with its word space $\mathcal W$. The coloring should help to understand how the smaller copies are orientated. They are not flipped over, but could be flipped.}
		\label{fig:diamond2}
	\end{figure}%
	We can quickly notice, that $\sim$ forms an equivalence relation since there are only two types of touching cells. Either the cells touch in $p$ (or iterations of it) or they touch in the line from top to bottom (and also in iterations of it). Both cases are harmless. For simplicity we only want to consider cells with word length $2$.\\
	In the first case we get that all copies intersect in $p$ and thus $\sim$ is transitive. Further we have for such a cell $w$ the property $R(w)=n$. In the other case we only have cells with $R(w)=2$, which immediately guarantees that $\sim$ is transitive, thus $\sim$ is an equivalence relation.\\
	Our aim is that this fractal fulfills \ref{B2} and the mass distribution is inhomogeneous. If we choose a mass distribution with $m(1), \dots, m(N)\in(0,1)$ arbitrarily and $\sum_{a=1}^Nm(a)=1$ this is in general wrong. It turns out, that some relations must hold on the mass distribution. For this we take a closer look at words of length $2$.\\
	First, let us consider the cells around $p$. Those cells are coded by $a1$ and in particular we want to consider two cells 
	\begin{equation*}
		w_1 = a_11\text{\qquad and\qquad}w_2=a_21
	\end{equation*}
	with $a, a_1, a_2\in\{\frac{n_0}2,\frac{n_0}2+1,n_0+1,\dots, N\}$.\\
	It holds that $w_1\sim w_2$ and $w_1^-\not\sim w_2^-$ is true. Thus $w_1$ and $w_2$ have to fulfill the first condition of \ref{B2} which implies 
	\begin{equation*}
		m(a_11) = m(w_1)= m(w_2) = m(a_21)
	\end{equation*}
	and further $m(a_1) = m(a_2)$.\\
	This has to be valid for all words around $p$ and we get, that 
	\begin{equation*}
		m(a) = m(N)
	\end{equation*}
	must hold for all $a\in\{\frac{n_0}2,\frac{n_0}2+1,n_0+1,\dots, N\}$.\\
	As a second case we consider the cells from top to bottom. Those cells are coded by 
	\begin{equation*}
		b1\text{\qquad and\qquad}bn_0
	\end{equation*}
	with $b\in\{1,\dots, n_0\}$.\\
	By construction it holds that two cells $w_1$ and $w_2$ intersect in the way, that 
	\begin{equation*}
		w_1=b_1d\text{\qquad and\qquad}w_2= (b_1+1)d
	\end{equation*}
	with $b_1\in\{1,\dots, n_0-2\}$ and $d\in\{1,n_0\}$.\\
	Again the first condition of \ref{B2} has to be fulfilled, which implies 
	\begin{equation*}
		m(b_1d) = m(w_1) = m(w_2) = m((b_1+1)d)
	\end{equation*}
	and thus 
	\begin{equation*}
	m(b_1) = m((b_1+1))
	\end{equation*}
	must hold. By induction it follows, that 
	\begin{equation*}
		m(b_1) = m(1)
	\end{equation*}
	must hold for $b_1\in\{1,\dots, n_0-1\}$.\\
	We are now nearly finished. The last two equivalent words we have to consider are
	\begin{equation*}
		w_1=(n_0-1)n_0\text{\qquad and\qquad}w_2=n_01.
	\end{equation*}
	By the previous results we have, that
	\begin{equation*}
		m(w_1) = m((n_0-1)n_0) = m((n_0-1))m(n_0) = m(1)m(n_0)
	\end{equation*}
	must hold. This implies also that
	\begin{equation*}
		m(w_1) = m(w_2)
	\end{equation*}
	holds independent of the choice of $m(1)$ and $m(n_0)$ and \ref{B2} is fulfilled.\\
	Thus we can choose 
	\begin{equation*}
		m(n_0)\in(0,1)\text{\qquad and\qquad}m(b_1) = \frac{1-m(n_0)}{N-1}
	\end{equation*}
	for $b_1\in\{1,\dots, n_0-1\}$ and at the same time \ref{B2} holds. This gives us an example with an inhomogeneous mass distribution and cells with $R(w)=n\in\mathbb N$.
\end{bsp}%
The examples of this section show some interesting facts and allow us to understand the condition \ref{B2} in a more deeper way. At the same time those examples demonstrate in a clear way that \ref{B2} is technically a proper condition, but in practice very complex.
%
% 
%----------------------------------------------------------------------------------------------------------------------------------------------------
%
%
\section{Analyze \ref{B2} with an algorithm}
\label{sec:algo}
As we have already seen in section \ref{sec:SG2}, we can create the Sierpi\'nski gasket with several various iterated function systems. Depending on the chosen function system, it varies, how many weights can be chosen and how many weights are fixed. This depends purely on the topology, while the attractor of the IFS stays the same although the distribution of the weights differ.\\
This fact raises the question how many weights we can choose, if we fix the attractor but vary at the same time the iterated function systems. At the same time we can ask how the number of free weights is distributed. It is clear that this is bounded by $N-1$ parameters, since the weights have to sum up to $1$ and thus the last chosen weight $m(a)$ has to be
\begin{equation*}
 m(a) = 1 - \sum_{i=1, i\neq a}^{N} m(i).
\end{equation*}
This implies that $N-1$ free weights are the optimal case and it is by now not clear, that this can be achieved for every attractor of an IFS (for the Sierpi\'nski gasket we already know this).\\
It seems to be impossible to answer this question in general for all possible fractals (or a wide class of fractals like all p.c.f.-fractals or all nested fractals) as this depends on each topology of every fractal. But whats about a single fractal? Can we answer this question for a given fractal positive? For the Sierpi\'nski gasket section \ref{sec:SG2} answers this question positive, since we found for every case an example. If we take a look at the 3-level Sierpi\'nski gasket SG$_3$ this question seems to be quite harder. We get a lot of equations which have to be fulfilled and we cannot solve them easily. So it seems to be quite hopeless to find easily suitable examples for every number of free weights despite the fact that they may not exist. The big problem is that we would have to check every possible IFS by hand and calculate how many free weights are possible. This may be possible for the Sierpi\'nski gasket where we would have to handle $6^3 = 216$ different iterated function systems, but if we take a look at the SG$_3$, this seems to be a big task since there are already $6^6 = \numprint{46656}$ possible iterated function systems. This problem gets even more harder, if we consider other fractals with more possible mappings or with a larger alphabet.\\
For this reason we want to develop a computer algorithm\footnote{implemented in \texttt{Python 3.6}, see\cite{Python}}, which calculates us the number of free parameters to all possible iterated function systems with the same attractor. Based on this we are also able to gather how the free parameters are distributed. 
\bigskip% 
\\%
The big idea of our algorithm is the following: we fix a particular fractal and put IFS, which are equivalent under rotating or flipping into an equivalence class. After this, we go through all equivalence classes and pick one representant. For each representant we set up the equations which have to be fulfilled. Since there are typically more equations than parameters, we simplify the equations up to a certain point. We then apply the implicit function theorem and receive as the number of free weights the dimension of the subspace, which we save (e.g. in a file). After this, we continue with the next IFS until we finished.%
\bigskip%
\\%
Our first step is to fix a certain fractal with $N$ similarities and an alphabet $\mathcal A = \{1,\dots, N\}$. This could be for example the Sierp\'nski gasket, the 3-level Sierpi\'nski gasket or any other fractal. We have to check by hand, if the equivalence relation is in fact equivalent according to Assumption \ref{A}. Luckily, we only have to check this for one particular realization, since all other realizations are permutations of the word space which won't effect the property of equivalence.\\
Further we have to check, if the homogeneous case solves the problem. This has it origin in \ref{B1}, where we stated that the Martin kernel in the homogeneous case exists. This has also checked only once, since the weights are equal and the equations stay the same under renaming of the variables.%
\bigskip%
\\%
In the next step we have to put all IFS into equivalence classes, since this saves a lot of calculating time. For a better understanding let us start with two small examples.
\begin{bsp}
	\label{bsp:SG2-mapped-1}
	Let us consider the Sierpi\'nski gasket. %
	\begin{figure}
		\centering%
		\includegraphics[page=24]{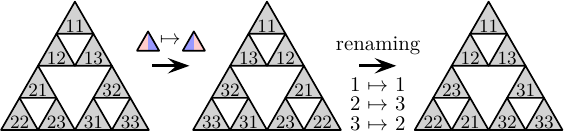}
		\caption{Consider a certain word space of the Sierpi\'nski gasket. We map this word space by flipping the whole fractal (middle picture). We can rename the word space in that way, that the top triangle is called ``1'', the left triangle ``2'' and the right triangle ``3''. With this we receive the right figure which represents an other word space as the starting word space.}
		\label{fig:SG2-mapped1}
	\end{figure}%
	\begin{figure}
		\centering%
		\includegraphics[page=25]{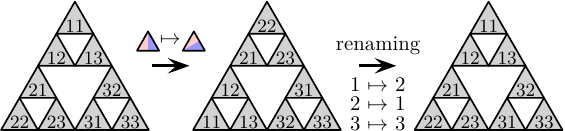}
		\caption{Consider a certain word space of the Sierpi\'nski gasket. We map this word space by flipping the whole fractal and rotate it by $\frac23\pi$ (middle picture). We can rename the word space in that way, that the top triangle is called ``1'', the left triangle ``2'' and the right triangle ``3''. With this we receive the right figure which represents the same word space we started with.}
		\label{fig:SG2-mapped2}%
	\end{figure}%
	We want to take a closer look at the iterated function system from Example \ref{bsp:SG2-1fp}. In this case the top triangle were shrank, the bottom left was also shrank and the bottom right triangle was rotated by $\frac23\pi$ and flipped. The word space (of the second depth) can be represented by the list
	\begin{equation*}
		[[11,12,13],[21,22,23],[32,33,31]].
	\end{equation*}
	and can also be found in Figure \ref{fig:SG2-mapped1}.\\
	Let us now consider what happens, if we just flip the whole fractal along the $y$-axis without any rotation. After this, the word space can be represented by
	\begin{equation*}
		[[11,13,12],[32,33,31],[21,23,22]].
	\end{equation*}
	We can rename the labels in such a way, that the top triangle starts with ``1'', the bottom left with ``2'' and the bottom right with ``3''. For this, we rename ``2'' to ``3'' and ``3'' to ``2''. The label of ``1'' stays the same. If we do so, we receive
	\begin{equation*}
		[[11,12,13],[23,22,21],[31,32,33]].
	\end{equation*}
	This is also illustrated in Figure \ref{fig:SG2-mapped1}. As we can see, we receive a different word space. The equations on this word space are however the same. For this reason we only have to find the number of free parameters of the first word space to receive the number of free parameters of the second word space.
\end{bsp}
For a deeper understanding let us consider another mapping of the whole word space.
\begin{bsp}
	\label{bsp:SG2-mapped-2}
	Again, we consider the Sierpi\'nski gasket and use the same word space 
	\begin{equation*}
		[[11,12,13],[21,22,23],[32,33,31]]
	\end{equation*}
	as in Example \ref{bsp:SG2-mapped-1} as starting word space. This time we rotate the fractal by $\frac23\pi$ and flip the fractal over. We receive
	\begin{equation*}
		[[22, 21, 23], [12, 11, 13],[31,32,33]].
	\end{equation*}
	Let us again rename the word space. We rename ``1'' to ``2'', ``2'' to "1" and ``3'' does not change. After renaming we receive
	\begin{equation*}
		[[11,12,13],[21,22,23],[32,33,31]]
	\end{equation*}
	and in Figure \ref{fig:SG2-mapped2} this is also illustrated. If we compare this with our starting word space, we see that they are identical. Thus this mapping will not generate an other representant in the equivalence class.\\
	Indeed this equivalence class consists in total of three different word spaces. The last missing one is
	\begin{equation*}
		[[11, 13, 12], [21, 22, 23], [31,32,33]]
	\end{equation*}
	which we would receive if we rotate the starting word space by $\frac23\pi$ but without flipping.\\
	From Example \ref{bsp:SG2-1fp} we then know that all these three different word spaces have exactly one free parameter.
\end{bsp}
Examples \ref{bsp:SG2-mapped-1} and \ref{bsp:SG2-mapped-2} give already a glue how we can put all word spaces into equivalence classes. We simply have to apply all possible rotations and flippings onto a starting word space and rename it afterwards. Since it is nearly impossible to predict which word spaces are put into equivalence classes we do the following: we go through all possible IFS and determine their equivalence class. After this, we only consider unique equivalence classes and drop all duplicates. This is also summed up in Listing \ref{lst:PseudoCodeEqClasses} as a pseudo algorithm. 
\begin{lstlisting}[style=pseudocode, caption={Pseudo algorithm for determine all equivalence classes}, label={lst:PseudoCodeEqClasses}]
*(set\_sw*) := *(set of all possible single word spaces*)
for *(single\_word\_space*) in *(set\_sw*):
    for *(each possible mapping*):
        *(mapped\_sw := apply the mapping onto single\_word\_space*)
        *(new\_sw := rename mapped\_sw*)
    *(save new\_sw as part of the equivalence class*)
*(determine all unique equivalence classes*)
\end{lstlisting}

In fact, we do not have to put the IFS into equivalence classes. But since the calculation of the free parameters takes a lot of computing time, this reduction is very welcome, since we only have to consider between 8-12$\%$ of all cases, which depends on the fractal. Further the calculation time for determine all equivalence classes is relatively short.

In a next step we want to go through all equivalence classes of iterated function systems. Since we do the same thing for every iterated function system, we want to fix a particular IFS and keep in mind that we want to iterate over all IFS. So, the next steps will be done for every IFS.\\
We can set up a list of equations, which have to be fulfilled regarding to Proposition \ref{prop:simplB2}. Every equation is of the form of
\begin{equation*}
	m(v_1)m(v_2)\cdots m(v_n) = m(w_1)m(w_2)\cdots m(w_n)
\end{equation*}
for equivalent words $v\sim w$ with $v= v_1\dots v_n$ and $w=w_1\dots w_n$.\\
This list of equations has to be replenished by
\begin{equation*}
	\sum_{a=1}^N m(a) - 1 = 0.
\end{equation*}
We can now solve those equations with a computer algebra system. For this we can use for example the \texttt{Python}-package \texttt{SymPy} (see \cite{SymPy}), which can solve also non--linear systems. Unfortunately \texttt{SymPy} returns on some IFS an error and says, that the list of equations is not solvable by its build--in function \texttt{nonlinsolve}. As a first consequence of this, we want to simplify the equations in parts on our own. For this, we have several possibilities, which we will discuss in the following methods.
\begin{method}[``delete double equations'']
	\label{method:dde}
	If two equations are identical, we can reduce our equation list by one of those equations and consider only the reduced equation list. For example we can reduce
	\begin{align*}
		&m(a) m(b) = m(c) m(d)\\
		&m(c) m(d) = m(a) m(b)
	\end{align*}
	to
	\begin{equation*}
		m(a) m(b) = m(c) m(d).
	\end{equation*}
\end{method}
This method seems to be quiet trivial, nevertheless one should mention this method. In addition this allows us later to reference to it and we may build up a more complex algorithm using this besides the following methods.
\begin{method}[``delete factorial variables'']
	\label{method:dfv}
	We can shorten any factor $m(i), i\in\mathcal A$, which occurs on both sides, since $m(i)>0$. For example we can replace
	\begin{equation*}%
		m(a)^km(b) = m(a)^lm(c)\text{ with } k,l\in\mathbb{R}
	\end{equation*}%
	by
	\begin{equation*}
		m(b)= m(a)^{l-k}m(c).
	\end{equation*}
\end{method}
This allows us to shorten unnecessary factors without touching the statement of the equation. This fact is essential, since we only want to simplify our equations for \texttt{SymPy}. The next method will also preserve the statement.
\begin{method}[``replace polynomial terms'']
	\label{method:rpt}
	We can extract roots or exponents, if it occurs on one side with the same exponent. This is possible, since $m(a) >  0$ for all $a\in\mathcal A$ must hold. %
	For example we can replace 
	\begin{equation*}
		m(a)^k = m(b)m(c)^2\text{ with }k\in\mathbb{R}\backslash\{0\}
	\end{equation*}
	by 
	\begin{equation*}
		m(a) = \left(m(b)m(c)^2\right)^{-k}
	\end{equation*}

\end{method}
The purpose of this method is the observation that \texttt{SymPy} has several problems with handling exponents, especially with rational exponents. This method seems to be counterproductive, since we modify the equations in the way, that (almost surely) rational exponents occur. The following method exploits another side effect, namely the occurrence of an isolated variable at one side.
\begin{method}[``substitute variables'']
	\label{method:sv}
	We can substitute variables, if they are isolated on one side and the exponent equals $1$ (which we can achieve by Method \ref{method:rpt}). For example we can replace
	\begin{align*}
		m(a)m(b)& = m(c)m(d)\\
		m(a) &= \sqrt{m(e)m(f)}
		\intertext{by}
		\sqrt{m(e)m(f)}m(b) &= m(c)m(d)\\
		m(a) &= \sqrt{m(e)m(f)}.
	\end{align*}
	Further we know that the substituted variable cannot be free and we reduced our problem by one dimension.
\end{method}
All these introduced methods are interesting, but alone they are not really useful. Thus we want to combine them and make them a very powerful weapon to simplify our equations.
\begin{method}[``simplify equations'']
	\label{method:se}
	We combine the Methods \ref{method:dde} - \ref{method:sv} to the following recursive method, called simplify equations.\\
	We start with all equations and call them free equations. Further we start with free variables, which are all variables. For the recursion we implement the fix equations and fix variables as empty.
	\begin{enumerate}
		\item\label{method:se:item1} replace in free equations all polynomial terms (by Method \ref{method:rpt})
		\item delete in free equations all  factorial variables (by Method \ref{method:dfv})
		\item delete in free equations all  double equations (by Method \ref{method:dde})
		\item if we can substitute a free variable in free equations (by Method \ref{method:sv}), substitute the variable and save it as a fix variable (and delete it from free variables). Further add the equation defining the variable as fix equation and delete it from free equations. After this, start again at \ref{method:se:item1}.\\
		Otherwise we are done.
	\end{enumerate}
	With this method we receive a list of free equations (potentially empty), a list of fix equations (defining the fix variables), a list of free variables and a list of fix variables.
\end{method}
After we apply Method \ref{method:se} to our equations we receive in most times an empty list of free equations. Further is the list of fix equations and the list of fix variables of the same length, since every equation in fixed equations provides a definition of a fixed variable. If the list of free equations is empty it is clear, that we can choose for every variable in the list of free variables a different value, where maybe some restrictions must hold. Nevertheless we can choose them independently from each other and those variables span up a subspace of $(0,1)^N$.\\
Since we are only interested in the dimension of this subspace (where the dimension equals the number of free parameters) we have to consider a mathematical way to verify that those free variables form a proper subspace.\\
For this, we want to apply the implicit function theorem. To do so, let us first fix some notations.
\begin{defi}
	\label{defi:nullstellenproblem}
	Let us fix a specific IFS, its word space and its starting equations. Let us apply Method \ref{method:sv} on the starting equations.\\
	For simplicity let us denote the free variables by $x_1,\dots x_k$ and the fixed variables by $y_1,\dots y_m$. These variables representing one specific $m(a)$ with $a\in\mathcal A$ and $k+m=N$ must hold. Further we want to note by $x\defgl (x_1,\dots, x_k)$ and $y\defgl (y_1,\dots, y_m)$ those variables combined as a vector.\\
	We can then write the list of all $M$ equations (the combination of free and fixed equations) as
	\begin{equation}
		\begin{gathered}
			g_1(x,y) = h_1(x,y)\\
			\vdots\\
			g_{M}(x,y) = h_{M}(x,y)
		\end{gathered}\label{eq:defi:nullstellenproblem-1}
	\end{equation}
	where $g_i(x,y)$ and $h_i(x,y)$ are polynomials in $x_1,\dots, x_k, y_1, \dots, y_m$\\
	We can define $F_i(x,y) \defgl g_i(x,y) - h_i(x,y)$ for $i=1,\dots, M$ and rewrite equations \eqref{eq:defi:nullstellenproblem-1} as a single function by
	\begin{equation*}
		F:(0,1)^k\times(0,1)^m \to \mathbb{R}^{M}, F(x,y) \defgl \begin{pmatrix}F_1(x,y)\\\vdots\\F_{M}(x,y)\end{pmatrix}
	\end{equation*}
	Finally we are interested in solutions with 
	\begin{equation*}
		F(x,y) = 0
	\end{equation*}
	and moreover to maximize $m$ and verify, that $m$ can be chosen maximal. $m$ is then the number of free parameters.
\end{defi}
With this notation it is much more clearer, what we want to do. Further this notation implies already an idea, in which direction our further considerations can go.\\
First of all, let ut notice that always $m \leq M$ holds. This comes from the fact, that for every variable $y_i$ an equation exists, namely in fix variables. Thus we can split up $M$ into $M = M_1 + m$, where $M_1$ equals the number of equations in free equations and $m$ the number of equations in fix equations.\\
If further $M_1=0$ and thus $M=m$ holds, we can apply the implicit function theorem.
\begin{prop}[Application of the implicit function theorem, cf. \cite{Forster}]
	\label{prop:applIFT}
	Consider a specific IFS. Let us denote by $F$ the simplified equations as in Definition \ref{defi:nullstellenproblem}. If $M=m$ holds, we are facing the following problem:
	\begin{equation*}
		F:(0,1)^k\times (0,1)^m\to\mathbb{R}^m, F(x,y)\defgl\begin{pmatrix}F_1(x,y)\\\vdots\\F_m(x,y)\end{pmatrix}
	\end{equation*}
	where $F$ is a continuous and differentiable function.\\
	Let $J$ be the Jacobian matrix defined by
	\begin{equation*}
		J(y)\defgl\frac{\partial F}{\partial y} \defgl \frac{\partial(F_1,\dots, F_m)}{\partial(y_1,\dots, y_m)}\defgl
		\begin{pmatrix}
			\frac{\partial F_1}{\partial y_1} & \cdots & \frac{\partial F_1}{\partial y_m}\\
			\vdots&\ddots&\vdots\\
			\frac{\partial F_m}{\partial y_1} & \cdots & \frac{\partial F_m}{\partial y_m}\\
		\end{pmatrix}
	\end{equation*}
	and define 
	\begin{equation*}
		\begin{gathered}
			x_{\text{hom}}\defgl\left(\frac1N, \dots, \frac1N\right)\text{ with $k$ entries,}\\
			y_{\text{hom}}\defgl\left(\frac1N, \dots, \frac1N\right)\text{ with $m$ entries.}
		\end{gathered}
	\end{equation*}
	If $\mathop{det} J\left(y_{\text{hom}}\right)\neq 0$ holds, we can choose $k$ variables arbitrarily in a small neighborhood $U_1\subset(0,1)^k$ of $x_{\text{hom}}$ and thus the IFS has $k$ free parameters.
\end{prop}
\begin{proof}
	The proof is in fact an application of the implicit function theorem.\\
	Recall that if the Jacobian $J$ is invertible at the point $y_{\text{hom}}$, then there exists an open neighborhood $U_1\subset(0,1)^k$ of $x_{\text{hom}}$, an open neighborhood $U_2\subset(0,1)^m$ of $y_{\text{hom}}$ and a continuous differentiable function $G:U_1\to U_2$ with $G(x_{\text{hom}}) =  y_{\text{hom}}$ such that
	\begin{equation*}
		F(x, G(x)) =0\text{\qquad holds for all }  x\in U_1.
	\end{equation*}
	The exact form of $G$ is for us irrelevant, but we can choose in $U_1$ each coordinate of $x$ independently from each other, thus we can choose $k$ weights and the IFS has $k$ free parameters. 
\end{proof}
\begin{rem}
	We can apply Proposition \ref{prop:applIFT} even if in the case of $m<M$. For this, we have to extend the fixed variables by $m_0 = M-m$ free variables. We have to be careful with the choice, since we could choose variables which lead to a non--invertible Jacobian matrix. Thus we have to consider in this case all possible subsets of free variables with cardinality $m_0$. We set up for each the Jacobian and if $\mathop{det} J(y_{\text{hom}})\neq 0$ we have found a constellation such that the implicit function theorem is applicable. It then follows, that the IFS has $N-M$ free parameters.
\end{rem}
In total we now know how to determine the number of free parameters. One thing still remains, which is the relatively high computing time.\\
We can speed up our algorithm if we split up our problem into multiple parts. To be precise, we split this up in the number of kernels of the computer and each kernel receives a part. Thus we can calculate parallel. A further acceleration would be possible, if we split up our algorithm to multiple computers, but we skip this here since it would make our algorithm harder to understand and the essential part of splitting up is already contained in the parallel computing--part.%
\bigskip 
\\
All together we can set up a pseudo algorithm, which can be found in Listing \ref{lst:PseudoCode}. % 
\begin{lstlisting}[style=pseudocode, caption={Pseudo algorithm for calculating the number of free parameters to all IFS generating a specific fractal.}, label={lst:PseudoCode}, columns=fullflexible, breaklines=true]
choose the fractal to consider
all_IFS := list of all possible IFS

all_eq_classes := determine all equivalence classes of all_IFS (compare to listing *(\ref{lst:PseudoCodeEqClasses}*))

for each kernel:
   pick an element from all_eq_classes which was uptil now not considered:
      set up list of equations associated to the IFS
      apply Method *(\ref{method:se}*) and receive free_equations, fix_equations, free_variable, fix_variables (with the notation of Definition *(\ref{defi:nullstellenproblem}*))
      if length(free_equations) = 0:
         set up Jacobian matrix J
         if det(J(y_hom)) *($\not=$*) 0:
            save M as number of free parameters
         else:
            raise Warning
      else:
         for all combinations of subsets of free_variable:
            combine combination and fix_variable
            set up Jacobian matrix J
            if det(J(y_hom)) *($\not=$*) 0:
               save M as number of free parameters
            else:
               raise Warning
merge results of the different kernels
return list with IFS, number of free parameters and length of its equivalence class
\end{lstlisting}%
In line 15 and 23 there is a statement for raising a Warning. In fact, up to now this has not occurred on any considered fractal. If a warning would raise we should extend our Method \ref{method:sv}.\\
This algorithm has some nice properties, but at the same time also some disadvantages. One of this disadvantages is the total time our algorithm needs for his total computation. This comes from the fact that we check all equivalence classes of iterated function systems. For this, recall that the number of all IFS grows very fast by the law
\begin{equation*}
 \# \text{IFS} = |\text{different mappings for a single copy}|^{N}.
\end{equation*}
The number of equivalence classes cannot calculated in an explicit form, but is approximately
\begin{equation*}
 \# \text{equivalence classes} \approx |\text{different mappings for a single copy}|^{N-1}.
\end{equation*}
Our algorithm needs more calculating time if the number of equivalence classes raises. Since the amount grows exponentially, this gets more worse, if either the alphabet gets bigger or if we have a higher number of possible small mappings. The fact, that we use parallel computing can only compensate this exponential growing rate in a minor way.\\
Further it should not be underestimated, that the number of equations raises also the computing time. Thus the time to terminate the algorithm depends in the first place on the number of IFS and in the second place on the numbers of equations/variables (and of course on the used hardware). 
\bigskip
\\
Let us now apply our algorithm to several (common) fractals. Of course, we want to investigate the already introduced fractals Sierpi\'nski gasket and the 3-level Sierpi\'nski gasket. Further we want to consider the Vicsek fractal (see \cite{Vicsek}), the Pentagasket with and without hole (see \cite{ASST2003, Imai2000}) and the Hexagasket (see \cite{Strichartz}). The Hexagasket could also be considered with a filled center, but for now we skip this, since it would exceed the computing time, since there are about 35.8mio different IFS and approximately 3mio equivalence classes.\\
As a small reminder contains Figure \ref{fig:attractors} (resp. the subfigures) the associated attractors of these fractals.\\%
\begin{figure}%
	\centering%
	\subcaptionbox{The Sierpi\'nski gasket.\label{fig:attractors-SG}}[0.29\textwidth]{\centering\includegraphics[page=26]{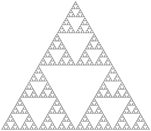}}%
	\quad%
	\subcaptionbox{The 3-level Sierpi\'nski gasket.\label{fig:attractors-SG3}}[0.29\textwidth]{\centering\includegraphics[page=27]{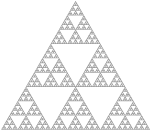}}%
	\quad%
	\subcaptionbox{The Vicsek fractal.\label{fig:attractors-Vicsek}}[0.29\textwidth]{\centering\includegraphics[page=28]{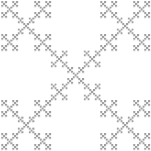}}%
	\\[3ex]%
	\subcaptionbox{The Pentagasket.\label{fig:attractors-Pentagasket}}[0.29\textwidth]{\centering\includegraphics[page=29]{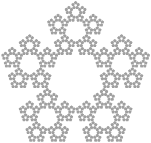}}%
	\quad%
	\subcaptionbox{The Pentagasket without a hole.\label{fig:attractors-PentagasketOhneLoch}}[0.29\textwidth]{\centering\includegraphics[page=30]{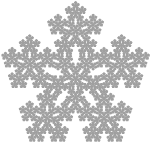}}%
	\quad%
	\subcaptionbox{The Hexagasket.\label{fig:attractors-Hexagasket}}[0.29\textwidth]{\centering\includegraphics[page=31]{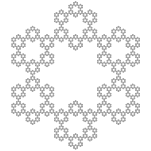}}%
	\caption{The attractors of the Sierpi\'nski gasket, the 3-level Sierpi\'nski gasket, the Vicsek fractal, the Pentagasket (with and without hole) and the Hexagasket.}%
	\label{fig:attractors}%
\end{figure}%
After we execute our algorithm we receive for every fractal a characterization of the associated space of solutions. In Table \ref{table:results-algo} the results are summarized. \\%
\begin{table}%
	\centering%
 	\begin{tabular}{@{}m{3.8cm}rrrrrR{2.0cm}@{}}
	\toprule& SG2 & SG3 & Vicsek & Pentagasket & \multicolumn1{m{1.99cm}}{Pentagasket without hole} & Hexagasket \\
	\midrule
	$N$& \numprint{3} & \numprint{6} & \numprint{5} & \numprint{5} & \numprint{6} & \numprint{6} \\[1ex]
	number of different mappings& \numprint{6} & \numprint{6} & \numprint{8} & \numprint{10} & \numprint{10} & \numprint{12} \\
	\addlinespace
	number of equations which have to be fulfilled& \numprint{3} & \numprint{9} & \numprint{4} & \numprint{5} & \numprint{25} & \numprint{6} \\
	\addlinespace
	number of equivalence classes& \numprint{44} & \numprint{7860} & \numprint{4360} & \numprint{10104} & \numprint{100220} & \numprint{250010} \\
	\addlinespace
	0 free parameter& \numprint{194} & \numprint{46257} & \numprint{20544} & \numprint{88025} & \numprint{999995} & \numprint{2599398} \\
	1 free parameter& \numprint{21} & \numprint{399} & \numprint{10112} & \numprint{11875} & \numprint{5} & \numprint{361007} \\
	2 free parameters& \numprint{1} & \numprint{0} & \numprint{2048} & \numprint{100} & \numprint{0} & \numprint{24075} \\
	3 free parameters& - & \numprint{0} & \numprint{64} & \numprint{0} & \numprint{0} & \numprint{1452} \\
	4 free parameters& - & \numprint{0} & \numprint{0} & \numprint{0} & \numprint{0} & \numprint{51} \\
	5 free parameters& - & \numprint{0} & - & - & \numprint{0} & \numprint{1} \\
	\addlinespace
	total number of IFS& \numprint{216} & \numprint{46656} & \numprint{32768} & \numprint{100000} & \numprint{1000000} & \numprint{2985984} \\
	\bottomrule
	\end{tabular}
	\caption{Summarized results about the number of free parameters on different fractals.}
	\label{table:results-algo}
\end{table}%
First of all the length of the alphabet $N$ is listed, which should be only a small reminder. Further the table contains the number of different small mappings for each small copy and the number of equations which have to be fulfilled such that \ref{B2} holds. Those equations differ for each IFS but the number stays the same since only the words interchange.\\
The next entry of the table contains the number of equivalence classes which occur. The amount of equivalence classes is slightly more than the number of different mappings to the power of $(N-1)$, since every mapping occurs at most once. But since there are some IFS that are invariant under mapping, there are some more.\\
The next few entries describe the occurrence of free parameters in total numbers. In other words, if we add all up, we receive the total number of IFS, which is as a reminder added at the bottom of the table.

As we can see, there is only one realization of the Sierpi\'nski gasket with two free parameters which is exactly Example \ref{bsp:SG2-2fp}. Further occur only 21 IFS with one free parameter and the majority of cases have zero free parameters.\\
The 3-level Sierpi\'nski gasket is generated by six small copies and thus five free parameters would be the maximum of possible free parameters. Indeed this does not occur. Neither four, three or two parameters do. As we can see, there are only 399 IFS  (which is about 0.9\%) with one free parameters and thus this is very uncommon. The remaining IFS have all zero free parameters which implies, that those IFS only fulfill \ref{B2} in the homogeneous case. One of the reasons for this fact is that there are in total nine equations which have to be fulfilled, making it harder to find any solution besides the homogeneous case.\\
On the Vicsek fractal the number of free parameters are wider distributed, for example there are 64 IFS with three free parameters and over \numprint{12160} IFS with one or two free parameters. Nevertheless an IFS where the full potential of four free parameters can be chosen is missing.\\
The Pentagasket behaves in a total other way than the Vicsek fractal, even if it has also an alphabet of length $N=5$. Most IFS have zero free parameters, about 11.8$\%$ have one free parameter and only \numprint{100} have two free parameters. Those \numprint{100} IFS are contained in \numprint{14} different equivalence classes.\\
If we add further a small copy in the middle, we receive the Pentagasket without hole. It is obvious, that this restricts the number of free parameters. Indeed this restricts the choice in a massive way. There is only one (!) equivalence class containing five IFS, where we can choose one free parameter. On all other IFS we are not able to choose any free parameter.\\
Last but not least we want to take a look at the Hexagasket. The distribution of free parameters is again relatively broad and we are indeed able to choose an IFS with the maximum of five free parameters. On the other hand we have again on most IFS choose only zero free parameters. This schema occurs on all fractals and we can discover another schema: if there are more equations which have to be fulfilled it is harder to choose many free parameters.\\
\\
The calculation time on the computer differs from fractal to fractal in a great manner. The following computing times refer all to the same computer with 60 kernels. The Sierpi\'nski gasket took about only a few seconds to compute all 44 equivalence classes. The 3-level Sierpi\'nski gasket SG$_3$ took about \numprint{1700} seconds, whereas the Vicsek only took \numprint{146} seconds and the Pentagasket took \numprint{563} seconds. Compared to the number of equivalence classes performs the SG$_3$ thus relatively slow.\\
The computing time of the Pentagasket without hole took extremely long and our computer needed about \numprint{191150} seconds (or approx. 2d 5h). The Hexagasket was again calculated relatively fast and took only \numprint{26400} seconds respectively 7h 20min.

In future we could use this algorithm to investigate also other fractals. For this we only need to code the needed equations and the definition of the small copies, which can be done relatively quickly. Also the determination of the equivalence classes can be done quickly. The most restricting part is indeed the calculation of the free parameters. The Pentagasket without hole indicates that the number of equations is the main factor for a long calculating time, which will be the limiting factor of further considerations.
%
% 
%----------------------------------------------------------------------------------------------------------------------------------------------------
% 
%
\bibliography{article}
\bibliographystyle{alpha}
\end{document}